\documentclass[11pt,a4paper]{amsart}

\usepackage{amsmath,amssymb,amsthm,mathtools}
\usepackage[T1]{fontenc}
\usepackage{enumitem}

\numberwithin{equation}{section}

\theoremstyle{plain}
\newtheorem{theorem}{Theorem}[section]

\newtheorem{lemma}[theorem]{Lemma}

\theoremstyle{definition}

\theoremstyle{remark}
\newtheorem{remark}[theorem]{Remark}

\DeclareMathOperator{\E}{\mathbb{E}}
\newcommand{\R}{\mathbb{R}}
\newcommand{\F}{\mathbb{F}}
\newcommand{\cF}{\mathcal{F}}
\newcommand{\Pbb}{\mathbb{P}}
\newcommand{\abs}[1]{\lvert#1\rvert}

\newcommand{\coloneq}{\mathrel{\mathop:}=}
\newcommand{\qv}[1]{\langle #1\rangle}
\newcommand{\qvc}[2]{\langle #1,#2\rangle}

\title[Carleson embedding with constant \(e\)]{A stochastic Carleson embedding theorem with constant \(e\)\\ and the vector of Riesz transforms}

\author[Domelevo]{Komla Domelevo}
\address{Institut f{\"u}r Mathematik, Julius-Maximilians-Universit{\"a}t
W{\"u}rzburg, Emil-Fischer-Str.~41, 97074 W{\"u}rzburg, Germany}
\email{komla.domelevo@uni-wuerzburg.de}

\author[Fladung]{Johanna Fladung}
\address{Institut f{\"u}r Mathematik, Julius-Maximilians-Universit{\"a}t
W{\"u}rzburg, Emil-Fischer-Str.~41, 97074 W{\"u}rzburg, Germany}
\email{johanna.fladung@uni-wuerzburg.de}

\author[Kakaroumpas]{Spyridon Kakaroumpas}
\address{Institut f{\"u}r Mathematik, Julius-Maximilians-Universit{\"a}t
W{\"u}rzburg, Emil-Fischer-Str.~41, 97074 W{\"u}rzburg, Germany}
\email{spyridon.kakaroumpas@uni-wuerzburg.de}

\author[Montobbio]{Paul Montobbio}
\address{Ecole Normale Supérieure, Paris, France}
\email{paul.montobbio@ens.psl.eu}

\author[Petermichl]{Stefanie Petermichl}
\address{Institut f{\"u}r Mathematik, Julius-Maximilians-Universit{\"a}t
W{\"u}rzburg, Emil-Fischer-Str.~41, 97074 W{\"u}rzburg, Germany}
\email{stefanie.petermichl@uni-wuerzburg.de}

\begin{document}

\begin{abstract}
We prove a continuous-time Carleson embedding theorem with constant \(e\) for a system of square-integrable continuous martingales whose quadratic covariations mimic the generalised Cauchy--Riemann relations. The argument is based on a Bellman function and a multidimensional It\^o formula. As an application we transfer the estimate to the upper half-space via the Gundy--Varopoulos representation and obtain a Carleson embedding, still with constant \(e\), for the vector consisting of a function and its Riesz transforms in arbitrary dimension.
\end{abstract}

\maketitle

\section{Introduction}

The classical Carleson embedding theorem asserts that a positive measure \(\mu\) on the upper half-space is Carleson if and only if the Poisson extension realises a bounded map from \(L^2(\R^n)\) into \(L^2(\mu)\). In the dyadic and martingale settings the same phenomenon is captured by a Carleson sequence (or process) and an embedding of the form
\[
\E\Bigl[\sum_I \alpha_I \abs{\E[f\mid \cF_I]}^2\Bigr]
\le
C\,\E[\abs{f}^2],
\]
with a vast literature on sharp constants going back to the Bellman function approach of Nazarov--Treil--Volberg \cite{NazTre1996,NV2004}. See also the book of Vasyunin--Volberg \cite{VasVol2020}.

The present note has two aims. First, we establish a continuous-time analogue for a \emph{system} of martingales rather than for a single scalar martingale. The driving processes \(X^0,\dots,X^n\) are assumed orthogonal and equivariant,
\[
\qvc{X^k}{X^\ell}=0\quad (k\neq \ell),
\qquad
\qv{X^k}=\qv{X^0},
\]
and the coordinate processes \(u^0,\dots,u^n\) are built from a common family of integrands by the linear combinations that formally reproduce the Cauchy--Riemann structure of a conjugate system in \(\R^{n+1}\). Under the continuous Carleson condition
\begin{equation}
\label{eq:C-intro}
M_t
\coloneq
\E\Bigl[\int_t^\infty \alpha_s\,ds\Bigm|\cF_t\Bigr]
\le 1
\qquad\text{a.s.,}
\end{equation}
and the assumption that $M$ has a similar structure with respect to the driving processes, we prove
\begin{equation}
\label{eq:CET-intro}
\E\Bigl[\int_0^\infty \alpha_s\sum_{k=0}^n (u^k_s)^2\,ds\Bigr]
\le
e\cdot
\E\Bigl[\sum_{k=0}^n (u^k_\infty)^2\Bigr].
\end{equation}
The constant \(e\) arises from the explicit Bellman function
\[
B(\tilde u^0,\dots,\tilde u^n,\tilde M,\tilde F)
=
e\tilde F-e^{1-\tilde M}\sum_{k=0}^n (\tilde u^k)^2
\]
on the natural state domain \(\sum (\tilde u^k)^2\le \tilde F\), \(0\le \tilde M\le 1\). After an application of It\^o calculus the second-order form reduces to an algebraic quadratic form in the integrands, which is nonnegative for every dimension \(n\ge 1\).

Secondly, we specialize the driving system to Brownian motion, or more precisely the Gundy--Varopoulos backgroung radiation process in the upper half-space \(\R^{n+1}_+=\R^n\times(0,\infty)\) constructed in~\cite{GV1979}. The coordinate processes of this process satisfy the orthogonality and equivariance hypotheses, and the Gundy--Varopoulos construction realizes each Riesz transform \(R_k f\) as the conditional expectation, given the exit point on \(\partial\R^{n+1}_+\), of a martingale of precisely the form \(u^k\) above. The Green potential of a density \(\alpha\) on \(\R^{n+1}_+\) supplies the representation of the Carleson martingale \(N\) required by the stochastic theorem. Consequently~\eqref{eq:CET-intro} becomes a Carleson embedding for the vector \((f,R_1f,\dots,R_nf)\), with the same constant \(e\) and with no dimensional loss.

The paper is organized as follows. Section~\ref{sec:stochastic} states and proves the stochastic Carleson embedding. Section~\ref{sec:GV} begins by recalling the Gundy--Varopoulos formula on the upper half-space and proceeds checking that the geometric hypotheses of Section~\ref{sec:stochastic} are satisfied. In particular, a variant of the Gundy--Varopolous representation on the upper half space away from the boundary is proved, see Lemma~\ref{lem:pairing}, which in contrast to the pointwise representation on the boundary it has the form of a pairing. Finally, Subsection~\ref{sec:riesz} combines the aforementioned results to draw the analytic corollary for the Riesz vector.

\section{The stochastic Carleson embedding}
\label{sec:stochastic}

\subsection{Setup and statement}

Let \((\Omega,\cF,\F,\Pbb)\) be a filtered probability space, with \(\F=(\cF_t)_{t\ge 0}\) satisfying the usual conditions (right-continuity and completeness; see~\cite{Klenke2020,KaratzasShreve1991}). Fix an integer \(n\ge 1\). For each \(k=0,\dots,n\), let \(X^k=(X^k_t)_{t\ge 0}\) be a continuous square-integrable real martingale, adapted to \(\F\), and assume the orthogonality and equivariance conditions
\begin{equation}
\label{eq:ortho}
\qvc{X^k}{X^\ell}=0
\quad\text{for all }k\neq \ell,
\qquad
\qv{X^k}=\qv{X^0}
\quad\text{for all }k=0,\dots,n.
\end{equation}
(The angle brackets denote the quadratic (co)variation in the sense of~\cite{KaratzasShreve1991}.) In addition, for each \(k=0,\dots,n\) let \(a^k=(a^k_t)_{t\ge 0}\) be a progressively measurable process such that
\[
\E\Bigl[\int_0^\infty \abs{a^k_s}^2\,d\qv{X^k}_s\Bigr]<\infty.
\]

Define the continuous \(L^2\)-bounded martingales
\begin{align}
\label{eq:u0}
u^0_t
&\coloneq
u^0_0
+\sum_{k=0}^n \int_0^t a^k_s\,dX^k_s,
\\
\label{eq:uk}
u^k_t
&\coloneq
u^k_0
+\int_0^t a^k_s\,dX^0_s
-\int_0^t a^0_s\,dX^k_s,
\qquad k=1,\dots,n,
\end{align}
where the initial values \(u^k_0\) are fixed random variables in \(L^2\), and the stochastic integrals are understood in the \(L^2\) theory of~\cite{Protter2003}. Boundedness in \(L^2\) implies almost-surely  the \(L^2\) convergence of each \(u^k\) to a terminal variable \(u^k_\infty\in L^2\) \cite{Klenke2020}.

Let \(\alpha=(\alpha_t)_{t\ge 0}\) be a nonnegative progressively measurable process. We assume the probabylistic analogue of the packing condition,
\begin{equation}
\label{eq:C}
M_t
\coloneq
\E\Bigl[\int_t^\infty \alpha_s\,ds\Bigm|\cF_t\Bigr]
\le 1
\qquad\text{a.s.}\; \text{for all }t\ge 0\ 
\tag{C}
\end{equation}
One may write
\[
M_t
=
N_t-\int_0^t \alpha_s\,ds,
\qquad
N_t
=
\E\Bigl[\int_0^\infty \alpha_s\,ds\Bigm|\cF_t\Bigr],
\]
so that \(N\) is the closed martingale associated with the total mass of \(\alpha\). We make the structural assumption that this martingale lies in the stable subspace generated by \(X^0,\dots,X^n\): namely
\begin{equation}
\label{eq:N-rep}
N_t
=
N_0
+\sum_{k=0}^n \int_0^t m^k_s\,dX^k_s,
\end{equation}
where \(N_0\in L^1\) and each \(m^k\) is progressively measurable and satisfies
\[
\E\Bigl[\int_0^t \abs{m^k_s}^2\,d\qv{X^k}_s\Bigr]<\infty
\qquad\text{for every }t\ge 0.
\]

\begin{theorem}
\label{thm:CET}
Under the assumptions above,
\begin{equation}
\label{eq:CET}
\E\Bigl[\int_0^\infty \alpha_s\sum_{k=0}^n (u^k_s)^2\,ds\Bigr]
\le
e\cdot
\E\Bigl[\sum_{k=0}^n (u^k_\infty)^2\Bigr].
\tag{CET}
\end{equation}
\end{theorem}

\subsection{Bellman function}

Let
\[
F_t
\coloneq
\E\Bigl[\sum_{k=0}^n (u^k_\infty)^2\Bigm|\cF_t\Bigr].
\]
Jensen's inequality \cite{Klenke2020} yields, for every \(t\ge 0\),
\[
\sum_{k=0}^n (u^k_t)^2
=
\sum_{k=0}^n \E[u^k_\infty\mid\cF_t]^2
\le F_t
\qquad\text{a.s.}
\]
Together with \(0\le M\le 1\) a.s., this suggests the state domain
\[
D
=
\Bigl\{
(\tilde u^0,\dots,\tilde u^n,\tilde M,\tilde F)=\tilde V\in\R^{n+3}
:
0\le \tilde M\le 1,\
\sum_{k=0}^n (\tilde u^k)^2\le \tilde F
\Bigr\}
\]
and the Bellman function
\[
B(\tilde V)
\coloneq
e\tilde F-e^{1-\tilde M}\sum_{k=0}^n (\tilde u^k)^2,
\qquad \tilde V\in D.
\]
It immediately follows that \(0\le B(\tilde V)\le e\tilde F\) on \(D\).

Write \(V_t=(u^0_t,\dots,u^n_t,M_t,F_t)\in D\). The map \(t\mapsto F_t\) needs not be continuous, so we apply It\^o's formula not to \(B(V_t)\) but to the continuous semimartingale
\[
\tilde B(W_t)
\coloneq
e^{1-M_t}\sum_{k=0}^n (u^k_t)^2,
\qquad
W_t=(u^0_t,\dots,u^n_t,M_t).
\]
The original Bellman function is then recovered by the elementary relation \(B(\tilde V)=e\tilde F-\tilde B(\tilde W)\).

\subsection{It\^o expansion}

Each \(u^k\) is a continuous martingale, and the Doob--Meyer decomposition of \(M\) reads \(M_t=N_t-\int_0^t\alpha_s\,ds\) \cite{KaratzasShreve1991}. The multidimensional It\^o formula \cite{KaratzasShreve1991,Protter2003} therefore gives
\begin{align*}
\tilde B(W_t)-\tilde B(W_0)
&=
-\int_0^t \partial_{\tilde M}\tilde B(W_s)\,\alpha_s\,ds
+\int_0^t \partial_{\tilde M}\tilde B(W_s)\,dN_s
\\
&\quad
+\sum_{k=0}^n\int_0^t \partial_{\tilde u^k}\tilde B(W_s)\,du^k_s
\\
&\quad
+\frac12\sum_{k=0}^n\int_0^t \partial_{\tilde u^k}\partial_{\tilde M}\tilde B(W_s)\,d\qvc{u^k}{N}_s
\\
&\quad
+\frac12\sum_{j=0}^n\int_0^t \partial_{\tilde M}\partial_{\tilde u^j}\tilde B(W_s)\,d\qvc{N}{u^j}_s
\\
&\quad
+\frac12\sum_{k,j=0}^n\int_0^t \partial_{\tilde u^k}\partial_{\tilde u^j}\tilde B(W_s)\,d\qvc{u^k}{u^j}_s
\\
&\quad
+\frac12\int_0^t \partial_{\tilde M}^2\tilde B(W_s)\,d\qv{N}_s.
\end{align*}
From the definition of \(\tilde B\) one has \(\partial_{\tilde u^k}\partial_{\tilde u^j}\tilde B=0\) whenever \(k\neq j\). Using the symmetry of the quadratic covariation we obtain
\begin{align*}
\tilde B(W_t)-\tilde B(W_0)
&=
-\int_0^t \partial_{\tilde M}\tilde B(W_s)\,\alpha_s\,ds
+\int_0^t \partial_{\tilde M}\tilde B(W_s)\,dN_s
\\
&\quad
+\sum_{k=0}^n\int_0^t \partial_{\tilde u^k}\tilde B(W_s)\,du^k_s
\\
&\quad
+\sum_{k=0}^n\int_0^t \partial_{\tilde u^k}\partial_{\tilde M}\tilde B(W_s)\,d\qvc{u^k}{N}_s
\\
&\quad
+\frac12\sum_{k=0}^n\int_0^t \partial_{\tilde u^k}^2\tilde B(W_s)\,d\qv{u^k}_s
\\
&\quad
+\frac12\int_0^t \partial_{\tilde M}^2\tilde B(W_s)\,d\qv{N}_s.
\end{align*}
Taking expectations, all martingale terms vanish and we are left with
\begin{equation}
\label{eq:expectation}
\E\bigl[\tilde B(W_t)-\tilde B(W_0)\bigr]
=
-\E\Bigl[\int_0^t \partial_{\tilde M}\tilde B(W_s)\,\alpha_s\,ds\Bigr]
+\frac12\E[H],
\end{equation}
where
\begin{equation}
\label{eq:H}
\begin{aligned}
H
&=
2\sum_{k=0}^n\int_0^t \partial_{\tilde u^k}\partial_{\tilde M}\tilde B(W_s)\,d\qvc{u^k}{N}_s
\\
&\quad
+\sum_{k=0}^n\int_0^t \partial_{\tilde u^k}^2\tilde B(W_s)\,d\qv{u^k}_s
+\int_0^t \partial_{\tilde M}^2\tilde B(W_s)\,d\qv{N}_s.
\end{aligned}
\end{equation}
The remainder of the proof consists in showing that \(H\ge 0\) almost surely.

\subsection{Quadratic (co)variations}

We record the brackets implied by~\eqref{eq:ortho} and by the definitions~\eqref{eq:u0}--\eqref{eq:uk} and~\eqref{eq:N-rep}. First,
\begin{align*}
\qv{u^0}_t
&=
\sum_{k,\ell=0}^n
\qvc{\int_0^\bullet a^k_s\,dX^k_s}{\int_0^\bullet a^\ell_s\,dX^\ell_s}_t
\\
&=
\sum_{k=0}^n\int_0^t (a^k_s)^2\,d\qv{X^k}_s
=
\int_0^t \sum_{k=0}^n (a^k_s)^2\,d\qv{X^0}_s,
\end{align*}
the cross terms vanishing by orthogonality and the last step using equivariance. Likewise,
\[
\qv{N}_t
=
\int_0^t\sum_{k=0}^n (m^k_s)^2\,d\qv{X^0}_s
\]
and
\[
\qvc{u^0}{N}_t
=
\int_0^t\sum_{k=0}^n a^k_s m^k_s\,d\qv{X^0}_s.
\]
For \(1\le k\le n\),
\begin{align*}
\qv{u^k}_t
&=
\qvc{\int_0^\bullet a^k_s\,dX^0_s-\int_0^\bullet a^0_s\,dX^k_s}{\,\cdot\,}_t
\\
&=
\int_0^t\bigl((a^0_s)^2+(a^k_s)^2\bigr)\,d\qv{X^0}_s,
\end{align*}
while
\begin{align*}
\qvc{u^k}{N}_t
&=
\int_0^t a^k_s m^0_s\,d\qv{X^0}_s
-\int_0^t a^0_s m^k_s\,d\qv{X^k}_s
\\
&=
\int_0^t\bigl(a^k_s m^0_s-a^0_s m^k_s\bigr)\,d\qv{X^0}_s.
\end{align*}

\subsection{The quadratic form}

The relevant derivatives of \(\tilde B\) are
\begin{align*}
\partial_{\tilde u^k}\tilde B
&=
2e^{1-\tilde M}\tilde u^k,
&
\partial_{\tilde u^k}^2\tilde B
&=
2e^{1-\tilde M},
\\
\partial_{\tilde M}\tilde B
&=
-e^{1-\tilde M}\sum_{j=0}^n (\tilde u^j)^2,
&
\partial_{\tilde M}^2\tilde B
&=
e^{1-\tilde M}\sum_{j=0}^n (\tilde u^j)^2,
\\
\partial_{\tilde u^k}\partial_{\tilde M}\tilde B
&=
-2e^{1-\tilde M}\tilde u^k.
\end{align*}
Substituting these expressions together with the brackets computed above, we find
\[
H
=
\int_0^t e^{1-M_s}\,Q_s\,d\qv{X^0}_s,
\]
where the process \(Q\) is given (omitting the time variable) by
\begin{equation}
\label{eq:Q}
\begin{aligned}
Q
&=
2\sum_{k=0}^n (a^k)^2
+2\sum_{k=1}^n\bigl((a^0)^2+(a^k)^2\bigr)
+\Bigl(\sum_{k=0}^n (u^k)^2\Bigr)\Bigl(\sum_{k=0}^n (m^k)^2\Bigr)
\\
&\qquad
-4u^0\sum_{k=0}^n a^k m^k
-4\sum_{k=1}^n u^k\bigl(a^k m^0-a^0 m^k\bigr).
\end{aligned}
\end{equation}
Since \(\qv{X^0}\) is a.s.\ increasing, the associated Lebesgue--Stieltjes measure it induces is nonnegative. Thus \(H\ge 0\) a.s.\ as soon as \(Q\ge 0\) a.s..

\subsection{Nonnegativity of \(Q\)}

Rearranging~\eqref{eq:Q} gives
\begin{align*}
Q
&=
2(n+1)(a^0)^2
-4a^0\Bigl(u^0 m^0-\sum_{k=1}^n u^k m^k\Bigr)
\\
&\quad
+\Bigl(\sum_{k=0}^n (u^k)^2\Bigr)\Bigl(\sum_{k=0}^n (m^k)^2\Bigr)
\\
&\quad
+4\sum_{k=1}^n (a^k)^2
-4\sum_{k=1}^n a^k\bigl(u^0 m^k+u^k m^0\bigr).
\end{align*}
We split the leading coefficient as \(2(n+1)=2(n-1)+4\) and complete squares. The last three terms on the second line already form perfect squares:
\begin{align*}
&4\sum_{k=1}^n (a^k)^2
-4\sum_{k=1}^n a^k\bigl(u^0 m^k+u^k m^0\bigr)
+\sum_{k=1}^n \bigl(u^0 m^k+u^k m^0\bigr)^2
\\
&\qquad=
\sum_{k=1}^n\bigl[2a^k-(u^0 m^k+u^k m^0)\bigr]^2.
\end{align*}
Next we use the Lagrange identity
\[
\Bigl(\sum_{k=0}^n c_k^2\Bigr)\Bigl(\sum_{k=0}^n d_k^2\Bigr)
-\Bigl(\sum_{k=0}^n c_k d_k\Bigr)^2
=
\sum_{0\le k<\ell\le n}(c_k d_\ell-c_\ell d_k)^2
\]
with the choice \(c_0=u^0\), \(c_k=-u^k\) for \(k\ge 1\), and \(d_k=m^k\). A short rearrangement produces
\begin{align*}
&\Bigl(\sum_{k=0}^n (u^k)^2\Bigr)\Bigl(\sum_{k=0}^n (m^k)^2\Bigr)
-\sum_{k=1}^n\bigl(u^0 m^k+u^k m^0\bigr)^2
\\
&\qquad=
\Bigl(u^0 m^0-\sum_{k=1}^n u^k m^k\Bigr)^2
+\sum_{1\le k<\ell\le n}(-u^k m^\ell+u^\ell m^k)^2.
\end{align*}
A further perfect square is
\begin{align*}
&4(a^0)^2
-4a^0\Bigl(u^0 m^0-\sum_{k=1}^n u^k m^k\Bigr)
+\Bigl(u^0 m^0-\sum_{k=1}^n u^k m^k\Bigr)^2
\\
&\qquad=
\Bigl[2a^0-\Bigl(u^0 m^0-\sum_{k=1}^n u^k m^k\Bigr)\Bigr]^2.
\end{align*}
Collecting terms we arrive at
\begin{equation}
\label{eq:Q-pos}
\begin{aligned}
Q
&=
2(n-1)(a^0)^2
+\sum_{k=1}^n\bigl[2a^k-(u^0 m^k+u^k m^0)\bigr]^2
\\
&\quad
+\Bigl[2a^0-\Bigl(u^0 m^0-\sum_{k=1}^n u^k m^k\Bigr)\Bigr]^2
+\sum_{1\le k<\ell\le n}(-u^k m^\ell+u^\ell m^k)^2.
\end{aligned}
\end{equation}
Every summand is nonnegative as soon as \(n\ge 1\). Hence \(Q\ge 0\) and \(H\ge 0\) almost surely.

\subsection{Conclusion}

Returning to~\eqref{eq:expectation} we obtain
\[
\E\bigl[\tilde B(W_t)-\tilde B(W_0)\bigr]
\ge
-\E\Bigl[\int_0^t \partial_{\tilde M}\tilde B(W_s)\,\alpha_s\,ds\Bigr]
\qquad\text{for all }t\ge 0.
\]
Since \(B(V)=eF-\tilde B(W)\) and \(F\) is a martingale (so \(\E[F_t]=\E[F_0]\)),
\begin{align*}
\E\bigl[B(V_0)-B(V_t)\bigr]
&=
\E\bigl[\tilde B(W_t)-\tilde B(W_0)\bigr]
\\
&\ge
\E\Bigl[\int_0^t e^{1-M_s}\Bigl(\sum_{k=0}^n (u^k_s)^2\Bigr)\alpha_s\,ds\Bigr].
\end{align*}
The bounds \(B\ge 0\) and \(M\le 1\) (hence \(e^{1-M}\ge 1\)) therefore yield
\[
\E\bigl[B(V_0)\bigr]
\ge
\E\Bigl[\int_0^t \alpha_s\sum_{k=0}^n (u^k_s)^2\,ds\Bigr]
\qquad\text{for all }t\ge 0.
\]
Letting \(t\to\infty\) and using \(B(\tilde V)\le e\tilde F\) together with the definition of \(F_0\), we conclude
\[
\E\Bigl[\int_0^\infty \alpha_s\sum_{k=0}^n (u^k_s)^2\,ds\Bigr]
\le
\E\bigl[B(V_0)\bigr]
\le
e\cdot\E\bigl[F_0\bigr]
=
e\cdot\E\Bigl[\sum_{k=0}^n (u^k_\infty)^2\Bigr],
\]
which is~\eqref{eq:CET}.

\begin{remark}
The term \(2(n-1)(a^0)^2\) in~\eqref{eq:Q-pos} is the only place where the dimension enters the algebraic identity. It is nonnegative precisely when \(n\ge 1\), which is the range relevant for Riesz transforms on \(\R^n\).
\end{remark}

\section{The Gundy--Varopoulos representation on the upper half-space}
\label{sec:GV}

Let \(f:\R^n\to\R\) be a Schwartz function, and let
\[
u(x,y)=P_yf(x)
\]
be its Poisson extension to the upper half-space
\begin{equation*}
    \R^{n+1}_+:=\{(x,y):x\in\R^n,\,y>0\}
\end{equation*}
For each \(k=1,\dots,n\) write \(R_kf\) for the \(k\)-th Riesz transform on \(\R^n\), with Fourier multiplier \(-i\xi_k/\abs{\xi}\), and let
\[
v_k(x,y)=P_y(R_kf)(x).
\]
When a single index is fixed we also write \(v=v_k\).

Let \(B=(B_t)_{t\ge 0}\) be Gundy--Varopoulos~\cite{GV1979} background radiation in \(\R^{n+1}_+\). Here, we only outline the basic properties of this stochastic process and do not give the details and instead refer the reader to~\cite{GV1979}. Thus, $B$ is Brownian motion with generator \(\tfrac12\Delta\), started from height \(a\to\infty\) with Lebesgue measure on the horizontal variables, and stopped at the first hitting time \(\tau\) of \(\{y=0\}\). We refer to~\cite{KaratzasShreve1991} for a comprehensive exposition on Brownian motion. The occupation formula for this process reads
\begin{equation}
    \label{eq:occupation}
    \E\int_0^\tau F(B_s)\,ds
=
\int_{\R^{n+1}_+} F(x,y)\,2y\,dy\,dx    
\end{equation}
for nonnegative Borel measurable functions \(F\) on $\R^{n+1}_{+}$, whenever either side is finite.

The coordinate processes of \(B\) realize a driving system as in Section~\ref{sec:stochastic}: they are pairwise orthogonal continuous martingales with a common quadratic variation \(t\wedge\tau\). Write \(\nabla=(\partial_y,\partial_{x_1},\dots,\partial_{x_n})\) for the full gradient and let \(A_k\) be the linear map on \(\R^{n+1}\) determined by
\[
A_k e_k=e_0,
\qquad
A_k e_0=-e_k,
\qquad
A_k e_\ell=0\text{ for }\ell\neq 0,k,
\]
where \(e_0\) is the vertical basis vector in the direction of \(y\). Set
\[
M_f(t)=\int_0^t \nabla u(B_s)\cdot dB_s,
\qquad
M_f^{(k)}(t)=\int_0^t A_k\nabla u(B_s)\cdot dB_s.
\]
These correspond precisely to the processes \(u^0\) and \(u^k\) of~\eqref{eq:u0}--\eqref{eq:uk}, with integrands \(a^0=\partial_y u(B)\) and \(a^\ell=\partial_{x_\ell}u(B)\). In the background-radiation limit one has \(u\to 0\) at infinity, hence the pathwise identity \(M_f(t)=u(B_t)\).

Let \(\varphi\) be a smooth real-valued function on \(\R^{n+1}_+\) with \(\Delta\varphi\ge 0\), and assume that the left-hand side of~\eqref{eq:riesz-CE} below is finite. Set \(\alpha_s=\Delta\varphi(B_s)\) along the background-radiation path. We assume that this density satisfies the Carleson condition~\eqref{eq:C}, and that the associated closed martingale \(N_t=\E\bigl[\int_0^\tau\alpha_s\,ds\bigm|\cF_t\bigr]\) admits the representation~\eqref{eq:N-rep} in the coordinate processes of \(B\). (If \(G(x,y)=\E^{x,y}\int_0^\tau\Delta\varphi(B_s)\,ds\) denotes the Dirichlet Green potential of \(\Delta\varphi\), then \(M_t=G(B_{t\wedge\tau})\) and the integrands of \(N\) are the components of \(\nabla G(B)\).)

\begin{theorem}
\label{thm:riesz-CE}
Under the assumptions above,
\begin{equation}
\label{eq:riesz-CE}
\begin{aligned}
&\int_{\R^{n+1}_+}
\Delta\varphi(x,y)\Bigl((P_yf)^2+\sum_{k=1}^n(P_y R_kf)^2\Bigr)\,2y\,dy\,dx
\\
&\qquad\le
e\cdot
\E\Bigl[M_f(\tau)^2+\sum_{k=1}^n M_f^{(k)}(\tau)^2\Bigr].
\end{aligned}
\end{equation}
\end{theorem}

The proof proceeds in two steps. First one compares the Poisson extensions \((u,v_k)\) with the Gundy--Varopoulos martingales \((M_f,M_f^{(k)})\) by an occupation pairing; that comparison is recorded as Lemma~\ref{lem:occupation} below (and is an equality when \(n=1\)). Then Theorem~\ref{thm:CET} is applied to the resulting right-hand side.

\begin{lemma}
\label{lem:occupation}
Let \(\varphi\) be smooth on \(\R^{n+1}_+\) with \(\Delta\varphi\ge 0\). Then, for each \(k=1,\dots,n\),
\begin{align*}
&\int_{\R^{n+1}_+}
\Delta\varphi(x,y)\Bigl((P_yf)^2+(P_y R_kf)^2\Bigr)\,2y\,dy\,dx
\\
&\qquad\le
\E\int_0^\tau
\Delta\varphi(B_s)\Bigl(M_f(s)^2+M_f^{(k)}(s)^2\Bigr)\,ds,
\end{align*}
whenever the left-hand side is finite. For \(n=1\) the two sides are equal. For \(n\ge 2\) the inequality is generally strict.
\end{lemma}

\subsection{Reduction}

By the occupation formula the left-hand side of Lemma~\ref{lem:occupation} equals
\[
\E\int_0^\tau
\Delta\varphi(B_s)\bigl(u(B_s)^2+v(B_s)^2\bigr)\,ds.
\]
The \(u\)-summands cancel against \(M_f(t)=u(B_t)\), and the claim is equivalent to
\[
\E\int_0^\tau\Delta\varphi(B_s)\,v(B_s)^2\,ds
\le
\E\int_0^\tau\Delta\varphi(B_s)\,M_f^{(k)}(s)^2\,ds.
\]
If \(n=1\), then \(R_1\) is the Hilbert transform and \((u,v)\) satisfy the Cauchy--Riemann system, so \(\nabla v=A_1\nabla u\) pointwise. It\^o for \(v\) then yields \(M_f^{(1)}(t)=v(B_t)\) pathwise, and the reduced comparison is an equality. The rest of the section treats general \(n\).

\subsection{A pairing identity}

Write \(V=A_k\nabla u\). Then \(\operatorname{div} V=0\) and the vertical component of \(V\) is \(V_y=\partial_{x_k}u\). The function \(v\) is harmonic, so It\^o likewise gives
\[
v(B_t)=\int_0^t\nabla v(B_s)\cdot dB_s
\]
in the same limit. The vector fields \(V\) and \(\nabla v\) are not equal when \(n\ge 2\): only their vertical components coincide. The comparison is obtained from a weaker statement, namely that \(M_f^{(k)}\) and \(v(B)\) have the same occupation pairing against functions of position.

\begin{lemma}[Gundy--Varopoulos representation on the upper half space]
\label{lem:pairing}
Let \(g\in C_c^\infty(\R^{n+1}_{+})\). Then
\[
\E\int_0^\tau M_f^{(k)}(s)\,g(B_s)\,ds
=
\E\int_0^\tau v(B_s)\,g(B_s)\,ds.
\]
\end{lemma}

\begin{proof}
We first work on the left-hand side.
Let \(U\) be the Dirichlet Green potential of \(g\) for the generator \(\tfrac12\Delta\), defined by
\[
U(z)=\E^z\int_0^\tau g(B_s)\,ds,
\]
with \(B\) Brownian motion of generator \(\tfrac12\Delta\), started at \(z\), killed when it hits \(\{y=0\}\). Equivalently,
\[
\tfrac12\Delta U=-g\quad\text{in }\R^{n+1}_+,
\qquad U=0\text{ on }\{y=0\}.
\]
Write \(d=n+1\) and let \(w^*\) be the reflection of \(w\) across \(\{y=0\}\).
The Dirichlet Green function of the half-space for \(\tfrac12\Delta\) is
\(G(z,w)=2\bigl(\Gamma(z-w)-\Gamma(z-w^*)\bigr)\), where \(\Gamma\) is the
free-space fundamental solution of \(\Delta\Gamma=-\delta\), that is, \(\Gamma(z)=c_d\abs{z}^{2-d}\) for \(d\ge 3\), and
\(\Gamma(z)=(1/2\pi)\log(1/\abs{z})\) for \(d=2\).
Thus
\[
U(z)=\int_{\R^{n+1}_+}G(z,w)\,g(w)\,dw.
\]
The support of \(g\) is compact in \(\{y>0\}\), so it is enough to expand
\(G(z,w)\) for \(w\) in a fixed compact set and \(\abs{z}\to\infty\).

If \(d\ge 3\), Taylor expansion of \(\Gamma\) about \(z\) gives
\[
\Gamma(z-w)-\Gamma(z-w^*)
=\nabla\Gamma(z)\cdot(w^*-w)+O\bigl(\abs{z}^{-d}\bigr).
\]
Here \(w^*-w=-2\eta\,e_0\) if \(w=(\xi,\eta)\), and
\(\nabla\Gamma(z)=c\,z/\abs{z}^d\), so the leading term is a multiple of
\(\eta y/\abs{z}^{n+1}\). Hence \(G(z,w)=O\bigl(y/\abs{z}^{n+1}\bigr)\), uniformly
for \(w\in\operatorname{supp}g\).

If \(d=2\), that is \(n=1\),
\[
G(z,w)=\frac1\pi\log\frac{\abs{z-w^*}}{\abs{z-w}}.
\]
The expansion
\(\log\abs{z-\zeta}=\log\abs{z}-z\cdot\zeta/\abs{z}^2+O(\abs{z}^{-2})\)
makes the two copies of \(\log\abs{z}\) cancel, and one is left with
\[
\log\frac{\abs{z-w^*}}{\abs{z-w}}
=\frac{2\eta y}{\abs{z}^2}+O\bigl(\abs{z}^{-2}\bigr).
\]
Again \(G(z,w)=O\bigl(y/\abs{z}^{2}\bigr)=O\bigl(y/\abs{z}^{n+1}\bigr)\), with no residual logarithm.
Differentiating the same expansions in \(z\) yields, for \(w\in\operatorname{supp}g\),
\[
\abs{\nabla_z G(z,w)}=O\bigl(\abs{z}^{-(n+1)}\bigr).
\]
Integrating against \(g\) we obtain
\[
U(z)=O\bigl(y/\abs{z}^{n+1}\bigr),
\qquad
\abs{\nabla U(z)}=O\bigl(\abs{z}^{-(n+1)}\bigr)
\qquad(\abs{z}\to\infty).
\]

We apply It\^o formula along a path to the process \(t\mapsto U(B_t)\) up to \(\tau\). Because the generator is \(\tfrac12\Delta\),
\[
dU(B_t)=\nabla U(B_t)\cdot dB_t+\tfrac12\Delta U(B_t)\,dt.
\]
Substitute \(\tfrac12\Delta U=-g\):
\[
dU(B_t)=\nabla U(B_t)\cdot dB_t-g(B_t)\,dt.
\]
Integrate from a time \(r\le\tau\) to \(\tau\):
\[
U(B_\tau)-U(B_r)
=\int_r^\tau\nabla U(B_s)\cdot dB_s-\int_r^\tau g(B_s)\,ds.
\]
On the boundary one has \(U=0\), and \(B_\tau\) is on the boundary, so \(U(B_\tau)=0\). Rearranging gives
\[
\int_r^\tau g(B_s)\,ds
=U(B_r)+\int_r^\tau\nabla U(B_s)\cdot dB_s.
\]

We now pair martingales. Recalling \(M_f^{(k)}(s)=\int_0^s V(B_r)\cdot dB_r\), with
\(V=A_k\nabla u\), stochastic Fubini \cite{Protter2003} and the identity above give
\begin{align*}
\E\int_0^\tau M_f^{(k)}(s)\,g(B_s)\,ds
&= \E\int_0^\tau\Bigl(\int_0^s V(B_r)\cdot dB_r\Bigr)g(B_s)\,ds \\
&= \E\int_0^\tau V(B_r)\cdot dB_r\cdot\Bigl(\int_r^\tau g(B_s)\,ds\Bigr)\\
&=\E\int_0^\tau V(B_r)\cdot dB_r\cdot
\Bigl(U(B_r)+\int_r^\tau\nabla U(B_s)\cdot dB_s\Bigr)\\
&=\E\int_0^\tau V(B_r)\cdot\nabla U(B_r)\,dr\\
&=\int_{\R^{n+1}_+} 2y\,V(x,y)\cdot\nabla U(x,y)\,dx\,dy.
\end{align*}

We now use the divergence theorem on a compact half-ball to write this last integral. Write \(z=(x,y)\) and
\[
\Omega_R=\bigl\{(x,y):y>0,\ \abs{z}<R\bigr\}.
\]
Write \(dw=dx\, dy\), then
\[
\int_{\Omega_R}2y\,V\cdot\nabla U\,dw
=-\int_{\Omega_R}U\,\operatorname{div}(2y\,V)\,dw
+\int_{\partial\Omega_R}U\,(2y\,V\cdot\nu)\,dS,
\]
where \(\nu\) is the outward unit normal to \(\Omega_R\). The boundary has two pieces.
On the bottom disk \(D_R=\{y=0,\ \abs{x}<R\}\),
the outward normal is \(\nu=-e_y\), so
\[
2y\,V\cdot\nu=2y\,V\cdot(-e_y).
\]
The prefactor \(y\) is identically zero on \(D_R\), hence this flux is zero.
On the hemisphere \(S_R=\{\abs{z}=R,\ y>0\}\),
the outward normal is \(\nu=z/\abs{z}\), so the flux is
\[
\int_{S_R}U(z)\,2y\,V(z)\cdot\frac{z}{\abs{z}}\,dS(z).
\]
That is the only remaining boundary term:
\[
\int_{\Omega_R}2y\,V\cdot\nabla U\,dw
=-2\int_{\Omega_R}U\,\operatorname{div}(yV)\,dw
+\int_{S_R}2y\,U\bigl(V\cdot\nu\bigr)\,dS.
\]
Let \(R\to\infty\). On \(S_R\) the bounds recorded above give
\(\abs{U}=O(y/R^{n+1})\). Since \(f\) is Schwartz, the Poisson extension \(u\)
satisfies \(\abs{\nabla u(z)}=O(\abs{z}^{-(n+1)})\), hence \(\abs{V}=O(R^{-(n+1)})\).
Taking into account that \(y\le R\) and the fact that the $n$-dimensional volume of $S_{R}$ is  \(O(R^n)\), we see that the integrand of the
hemisphere flux is \(O(R^{-n})\), so the integral tends to \(0\).

Therefore
\[
\int_{\R^{n+1}_+}2y\,V\cdot\nabla U\,dw
=-2\int_{\R^{n+1}_+}U\,\operatorname{div}(yV)\,dw
=-2\int_{\R^{n+1}_+}U\,\partial_{x_k}u\,dw,
\]
since \(\operatorname{div} V=0\) and \(V_y=\partial_{x_k}u\), so \(\operatorname{div}(yV)=\partial_{x_k}u\). In summary, for the left-hand side we obtain
\[
\E\int_0^\tau M_f^{(k)}(s)\,g(B_s)\,ds=-2\int_{\R^{n+1}_+}U\,\partial_{x_k}u\,dw.
\]

We turn to the right-hand side. The occupation formula~\ref{eq:occupation} and \(\tfrac12\Delta U=-g\) give
\[
\E\int_0^\tau v(B_s)\,g(B_s)\,ds
=\int_{\R^{n+1}_+} 2y\,v\bigl(-\tfrac12\Delta U\bigr)\,dw
=-\int_{\R^{n+1}_+} yv\,\Delta U\,dw.
\]
Green's identity for the pair \((U,yv)\) on the same truncated region \(\Omega_R\) reads
\[
\int_{\Omega_R}\bigl(U\,\Delta(yv)-yv\,\Delta U\bigr)\,dw
=\int_{\partial\Omega_R}\bigl(U\,\partial_\nu(yv)-yv\,\partial_\nu U\bigr)\,dS.
\]
On the bottom disk \(D_R=\{y=0,\,\abs{x}<R\}\) one has \(U=0\) and \(yv=0\) as traces, so the integrand vanishes.
On the hemisphere \(S_R\), the same bounds give
\(\abs{U}=O(R^{-n})\) and \(\abs{\nabla U}=O(R^{-(n+1)})\).
Since \(f\) is Schwartz, \(R_kf(x)=O(\abs{x}^{-n})\) at infinity, and the
Poisson extension therefore satisfies
\(\abs{v(z)}+ \abs{z}\cdot\abs{\nabla v(z)}=O(\abs{z}^{-n})\).
Thus \(\abs{yv}=O(R^{1-n})\) and \(\abs{\nabla(yv)}=O(R^{-n})\) on \(S_R\).
The integrand \(U\,\partial_\nu(yv)-yv\,\partial_\nu U\) is \(O(R^{-2n})\),
the $n$-dimensional area is \(O(R^n)\), and the flux tends to \(0\) as \(R\to\infty\).

Therefore
\[
\int_{\R^{n+1}_+}yv\,\Delta U\,dw
=\int_{\R^{n+1}_+}U\,\Delta(yv)\,dw.
\]
The product rule and harmonicity of \(v\) give \(\Delta(yv)=2\partial_yv\). Combined with occupation and \(\tfrac12\Delta U=-g\), we summarize for the right-hand side
\[
\E\int_0^\tau v(B_s)\,g(B_s)\,ds
=-2\int_{\R^{n+1}_+}U\,\partial_yv\,dw.
\]

To compare the left- and right-hand sides, note the identity \(\partial_yv=\partial_{x_k}u\), which is immediately seen to be true by checking it on the Fourier side. Indeed: with the multiplier \(\widehat{R_kf}(\xi)=-i\xi_k/\abs{\xi}\,\hat f(\xi)\) one has
\[
\hat u(\xi,y)=e^{-y\abs{\xi}}\hat f(\xi),
\qquad
\hat v(\xi,y)=e^{-y\abs{\xi}}\Bigl(-i\frac{\xi_k}{\abs{\xi}}\Bigr)\hat f(\xi),
\]
hence
\[
\widehat{\partial_{x_k}u}(\xi,y)=i\xi_k\,e^{-y\abs{\xi}}\hat f(\xi)
=\widehat{\partial_yv}(\xi,y).
\]
Equivalently, if \(Q_y^{(k)}\) denotes the conjugate Poisson kernel, then \(\partial_{x_k}P_y=\partial_y Q_y^{(k)}\) pointwise, and \(v=Q_y^{(k)}*f\) is the choice of conjugate that matches the matrix \(A_k\).
\end{proof}

\begin{remark}
That identity is the only Cauchy--Riemann relation used. The remaining slots of \(\nabla v\) do not match those of \(A_k\nabla u\) when \(n\ge 2\), and are never invoked.
\end{remark}

\subsection{Conclusion of the proof}

Let \(0\le\eta_{\ell}\leq 1\) be smooth, compactly supported cutoff functions on $\R^{n+1}_{+}$ that increase pointwise to $1$, and set \(g_\ell=\eta_\ell\Delta\varphi\,v\). Lemma~\ref{lem:pairing} applied to each \(g_\ell\) yields
\[
\E\int_0^\tau\eta_\ell\Delta\varphi(B_s)\,v(B_s)\bigl(M_f^{(k)}(s)-v(B_s)\bigr)\,ds
=0.
\]
Expanding the squares,
\[
\E\int_0^\tau\eta_\ell\Delta\varphi\bigl(M_f^{(k)\,2}-v(B)^2\bigr)\,ds
=\E\int_0^\tau\eta_\ell\Delta\varphi\bigl(M_f^{(k)}-v(B)\bigr)^2\,ds
\ge 0,
\]
hence
\[
\E\int_0^\tau\eta_\ell\Delta\varphi\,v(B)^2\,ds
\le
\E\int_0^\tau\eta_\ell\Delta\varphi\,M_f^{(k)\,2}\,ds
\le
\E\int_0^\tau\Delta\varphi\,M_f^{(k)\,2}\,ds.
\]
Monotone convergence on the left, using that the left-hand side of the original display is finite, gives the reduced comparison. Combined with \(M_f=u(B)\), this is Lemma~\ref{lem:occupation}.

\begin{remark}
    If \(\Delta\varphi\) is itself compactly supported in \(\{y>0\}\), the cutoff is unnecessary. The hypothesis that the left-hand side is finite may be replaced by the assumption that both sides are well-defined in \([0,+\infty]\); if the right-hand side is infinite there is nothing to prove.
\end{remark}

\subsection{Adding all Riesz transforms}
\label{sec:riesz}

This consists in applying the pairing lemma to each index separately and adding. We explain the details.
Write \(v_k=P_y(R_kf)\) and \(V^{(k)}=A_k\nabla u\). For every \(k\) one still has \(\operatorname{div} V^{(k)}=0\), \(V^{(k)}_y=\partial_{x_k}u\), and the same Cauchy--Riemann identity
\[
\partial_y v_k=\partial_{x_k}u.
\]
Lemma~\ref{lem:pairing} therefore holds for each \(k\):
\[
\E\int_0^\tau\bigl(M_f^{(k)}(s)-v_k(B_s)\bigr)\,g(B_s)\,ds=0
\]
for all \(g\in C_c^\infty(\R^{n+1}_+)\).
Taking \(g=\eta_\ell\Delta\varphi\,v_k\) and summing in \(k\) produces no mixed terms \(M_f^{(k)}v_\ell\) with \(k\neq \ell\). Hence
\begin{align*}
&\E\int_0^\tau\Delta\varphi(B_s)\sum_{k=1}^n\bigl(M_f^{(k)}(s)^2-v_k(B_s)^2\bigr)\,ds
\\
&\qquad=
\E\int_0^\tau\Delta\varphi(B_s)\sum_{k=1}^n\bigl(M_f^{(k)}(s)-v_k(B_s)\bigr)^2\,ds
\ge 0.
\end{align*}
Together with \(M_f(t)=u(B_t)\) this is
\begin{equation}
\label{eq:vector-occupation}
\begin{aligned}
&\int_{\R^{n+1}_+}
\Delta\varphi\Bigl((P_yf)^2+\sum_{k=1}^n(P_y R_kf)^2\Bigr)\,2y\,dy\,dx
\\
&\qquad\le
\E\int_0^\tau
\Delta\varphi(B_s)\Bigl(M_f(s)^2+\sum_{k=1}^n M_f^{(k)}(s)^2\Bigr)\,ds,
\end{aligned}
\end{equation}
under the same hypotheses as Lemma~\ref{lem:occupation}, that is, \((f\) is  Schwartz function On $\R^n$, \(\Delta\varphi\ge 0\), and the left-hand side finite.
No other changes are needed in the argument. In particular one still does not need \(\nabla v_k=A_k\nabla u\), nor any relation between \(\sum_k\abs{A_k\nabla u}^2\) and \(\sum_k\abs{\nabla v_k}^2\). Those quadratic forms are actually not equal for \(n\ge 2\): \(\sum_k\abs{A_k\nabla u}^2=\abs{\nabla_x u}^2+n(\partial_y u)^2\).

The density \(\alpha_s=\Delta\varphi(B_s)\) satisfies~\eqref{eq:C} by assumption, and \(N\) admits a representation of the form~\eqref{eq:N-rep} with respect to the coordinate processes of \(B\). The right-hand side of~\eqref{eq:vector-occupation} is therefore in the shape of the left-hand side of~\eqref{eq:CET} for the system \(M_f,M_f^{(1)},\dots,M_f^{(n)}\). Theorem~\ref{thm:CET} yields~\eqref{eq:riesz-CE}, which is exactly Theorem~\ref{thm:riesz-CE}.

\subsection{Remarks}
A few remarks are in order.

\begin{enumerate}
    \item \textbf{Equality.} For \(n=1\) one has \(\nabla v=A_1\nabla u\), so \(M_f^{(1)}(t)=v(B_t)\) pathwise and the original display is an equality. For \(n\ge 2\) the integrand \(\bigl(M_f^{(k)}-v(B)\bigr)^2\) is not identically zero for generic \(f\), and the inequality is strict for generic \(\varphi\).

    \item \textbf{What the pairing does not say.} Lemma~\ref{lem:pairing} asserts that \(M_f^{(k)}\) and \(v(B)\) have the same occupation pairing against functions of position. It does not assert the process identity \(\E\bigl[M_f^{(k)}(s)\bigm|B_s\bigr]=v(B_s)\), and that identity is not used.

    \item \textbf{What is not used.} The argument does not identify \(\nabla v\) with \(A_k\nabla u\). It does not apply the tower property to the non-nested pair of $\sigma$-algebras \(\sigma(B_s)\) and \(\sigma(B_\tau)\). It does not treat \(W=M_f^{(k)}(\tau)-R_kf(B_\tau)\) as a functional of the future path after time \(s\). It does not introduce an occupation density \(\mu\) and invoke Liouville uniqueness for \(y(\mu-v)\).

    \item \textbf{Finite starting height.} At finite starting height \(a\), occupation is given by the density \(2\min(y,a)\,dx\,dy\) rather than \(2y\,dx\,dy\), and \(M_f(t)=u(B_t)-u(B_0)\). For test functions supported in the truncated domain \(\R^n\times\{y\le a/2\}\) the occupation converges to \(2y\), and the missing incoming integral from infinity to height \(a\) is exponentially small in \(L^2\) for Schwartz \(f\). This explains why we passed to the \(a\to\infty\) limit in which the occupation formula with density \(2y\) holds.

    \item \textbf{Sings.} The matrix \(A_k\) is the one for which \(A_k\nabla u\cdot dB=\partial_{x_k}u\,dY-\partial_y u\,dX^k\), matching~\eqref{eq:uk}. With that choice, \(\partial_yv=\partial_{x_k}u\) holds as written. The same computation with \(V=\nabla u\) recovers the tautology \(M_f=u(B)\).
\end{enumerate}

\section*{AI disclosure statement}
We have used Grok 4.6, built by xAI, to accelerate the write-up, and to help for Lemma~\ref{lem:pairing}.
Otherwise all ideas, models and the strategy of proofs are ours.

\end{document}